\documentclass[12pt]{article}
\usepackage{amsmath,amsfonts,amssymb, amsthm, rotating,colordvi}
\usepackage{colordvi}
\def\qed{\nopagebreak\hfill{\rule{4pt}{7pt}}}
\def\proof{\noindent {\it{Proof.} \hskip 2pt}}

\newtheorem{theo}{Theorem}[section]

\newtheorem{lemm}[theo]{Lemma}
\newtheorem{prop}[theo]{Proposition}

\numberwithin{equation}{section}

\newdimen\Squaresize \Squaresize=11pt
\newdimen\Thickness \Thickness=0.7pt
\def\Square#1{\hbox{\vrule width \Thickness
   \vbox to \Squaresize{\hrule height \Thickness\vss
    \hbox to \Squaresize{\hss#1\hss}
   \vss\hrule height\Thickness}
\unskip\vrule width \Thickness} \kern-\Thickness}

\def\Vsquare#1{\vbox{\Square{$#1$}}\kern-\Thickness}

\def\moins{\raise 1pt\hbox{{$\scriptstyle -$}}}

\title{A combinatorial identity on Galton-Watson process}
\author{Linyuan Lu
\thanks{University of South Carolina, Columbia, SC 29208,
({\tt lu@math.sc.edu}). This author was supported in part by NSF
grant DMS 1300547 and ONR grant N00014-13-1-0717.}
\and
Arthur L.B. Yang
\thanks{Center for Combinatorics, LPMC-TJKLC, Nankai University,
Tianjin 300071, P. R. China ({\tt yang@nankai.edu.cn}).
}
}

\begin{document}
\maketitle








\begin{abstract}
 Let $f(m,c)=\sum_{k=0}^{\infty} (km+1)^{k-1} c^k e^{-c(km+1)/m} / (m^kk!)$.
For any positive integer $m$ and positive real $c$,
the identity $f(m,c)=f(1,c)^{1/m}$ arises in the random
graph theory. In this paper, we present two elementary proofs of this identity:
a pure combinatorial proof and a
power-serial proof. We also proved
that this identity holds for
any positive reals $m$ and $c$. 
\end{abstract}

\section{Introduction}
Erd\H{o}s and R\'enyi wrote a series of remarkable papers on the
evolution of random graphs around 1960 \cite{ER1,ER2}.
Erd\H{o}s and R\'enyi first considered the uniform model $G_{n,f}$ 
where a random graph $G$ is selected 
uniformly among all graphs with $n$ vertices and $f$
edges. It is remarkable that phase transition happens as $f$ passes
through the threshold $f\approx n/2$. It is convenient to write 
$f=cn/2$ where $c$ is the average degree. When $0<c<1$, almost surely
all connected components are of order $O(\ln n)$.  When $c=1$, the largest component has the order of $\Theta(n^{2/3})$.
When $c>1$,  almost
surely there is a unique giant component of order $(g(c)+o(1))n$, where
$$g(c)=1-\sum_{k=0}^\infty\frac{(k+1)^{k-1}c^{k}}{k!}e^{-(k+1)c}.$$
Here the term $\frac{(k+1)^{k-1}c^{k}}{k!}e^{-(k+1)c}$ in the summation
above is the probability that a random vertex belongs to a connected
component of order $k+1$.

For a positive integer $m$ and a positive real $c$, define
$$f(m,c)=\sum_{k=0}^{\infty} \frac{(km+1)^{k-1} c^k}{m^k k!} e^{-c(km+1)/m}.$$
Observe that $f(1,c)=1-g(c)$. So $f(1,c)$ measures the
probability that a vertex $v$ is in small components. This probability
can be computed using the {\em Branching
Process}, which reveals the neighbors $v$, and the neighbors of its
neighbors iteratively, until the whole component containing $v$ is
revealed.  This branching process can be coupled by the {\em Poisson
  Process}, a special case of a Galton-Watson process. In general,
a Galton-Watson process is a stochastic process $\{Y_t\}_{t=0}^{\infty}$ which
evolves according to the recursive formula
\begin{equation}
  \label{eq:gw}
Y_{t} = Y_{t-1} -1 + Z_t,  \quad \mbox{ for } t\geq 1,
\end{equation}
where $Y_0 = 1$. The {\em Poisson Process} is a special Galton-Watson
process with $Z_t$ ensembling the Poisson distribution
($\Pr(Z_t=k)=c^ke^{-c}/k!$ for $k=0,1,2\ldots$).
For simplicity, let $z=f(1,c)$, be the probability that the Poisson
Process stops after finite steps. Suppose that the root node has $k$
children nodes. The subprocesses starting at each child node are
independent to each other, and are identical to the main process on
the distribution. Thus the probability that a subprocess terminates in finite steps
is also $z$. By the independency, all $k$ subprocesses terminates in
some finite steps is exactly $z^k$. This leads the following recursive
formula for $z$:
\begin{align}
\nonumber
  z&=\sum_{k=0}^\infty \Pr(Z_1=k)z^k\\
\nonumber
&=\sum_{k=0}^\infty \frac{c^ke^{-c}}{k!}z^k\\
&=e^{-c(1-z)}.
\label{eq:1}
\end{align}

In summary, we have the following proposition from the random graph theory.
\begin{prop}
If $0\leq c\leq 1$, then
 $f(1,c)\equiv 1$. If $c>1$, then 
 $f(1,c)$ is the
unique root of Equation \eqref{eq:1} in the interval $(0,1)$.
\end{prop}

Erd\H{o}s and R\'enyi's classical work on random graphs can be
generalized to random hypergraphs. Let $H^r_{n,f}$ be the
random $r$-uniform hypergraph such that each $r$-uniform hypergraph
with $n$ vertices and $f$ edges is selected with the same probability.
The phase transition of the random hypergraph $H^r_{n,f}$ was
analyzed by Schmidt-Pruzan-Shamir \cite{PS} and  Karo\'nski-{\L}uczak
\cite{KL}. Namely, they proved
that for $f< \tfrac{n}{r(r-1)}$, almost surely the largest connected component
in $H^r_{n,f}$ is of size $O(\ln n)$; for $f \sim \tfrac{n}{r(r-1)}$, almost
surely the largest connected component is of order $\Theta(n^{2/3})$; for
$f>\tfrac{n}{r(r-1)}$, almost surely there is a unique giant connected
component.
Write $f=c\frac{r(r-1)}{n}$ so that $c=1$ is the threshold.
Then the order of the unique giant connected component is
$(g_r(c)+o_n(1))n$, where
$$g_r(c)=1-\sum_{k=0}^{\infty} \frac{(k(r-1)+1)^{k-1} c^k}{(r-1)^k k!} e^{-(kr-k+1)c}.$$
Setting $m=r-1$, we observe that $g_r(c)=1-f(m,c)$. So $f(m,c)$ is the
probability that a vertex $v$ belongs to small components in
$H^r_{n,f}$ (with $r=m+1$ and $f=c\tfrac{n}{(m+1)m}$).

Similarly, the branching process on $H^r_{n,f}$ can be coupled by
an {\em $m$-fold Poisson Process}: a special  Galton-Watson
process with $Z_t$ ensembling the following $m$-fold Poisson distribution
with the expected value $c$:
$$\Pr(Z_t=k)=
\begin{cases}
e^{-c}\frac{(c/m)^{k/m}}{(k/m)!}   & \mbox{ if } k \mbox{ is a multiple of
  } m\\
0 & \mbox{ otherwise.}
\end{cases}
$$
Let $y=f(m,c)$. A similarly argument shows that $y$ satisfies the
following recursive formula:
\begin{align}
\nonumber
  y&=\sum_{k=0}^\infty \Pr(Z_1=k)y^k\\
\nonumber
&=\sum_{j=0}^\infty \Pr(Z_1=mj)y^{mj}\\
\nonumber
&=\sum_{j=0}^\infty e^{-c}\frac{(c/m)^{j}}{j!}y^{mj}\\
\label{eq:2}
&=e^{-c(1-y^m)/m}.
\end{align}

This leads the following proposition for $f(m,c)$:
\begin{prop}
Suppose that $m$ is a positive integer.
If $0\leq c\leq 1$, then
 $f(m,c)\equiv 1$. If $c>1$, then 
 $f(m,c)$ is the
unique root of Equation \eqref{eq:2} in  the interval $(0,1)$.
\end{prop}


From equation \eqref{eq:2}, we have
\begin{equation}
  \label{eq:3}
y^m=e^{-c(1-y^m)}.  
\end{equation}
Setting $z=y^m$, then $y$ satisfies Equation \eqref{eq:1}.
Thus we have the following identity.
\begin{theo}\label{main}
For any integer $m\geq 1$ and any real $c>0$, we have
$$f(m,c)^m=f(1,c).$$
\end{theo}

The goal of this paper is to give two elementary proofs of  this identity without using any
random graph theory. In section 2, we will give a pure combinatorial
proof. In section 3, we will give a power-series proof. In the last
section, we will extend this identity to any real $m>0$.

\section{Useful Lemmas}
\subsection{Convergence of $f(m,c)$}
\begin{lemm}
  For any positive reals $m$ and $c$, the series in the definition of
  $f(m,c)$ converges. Moreover, $f(m,c)$ is a continuous function.
\end{lemm}
\proof
We can rewrite $f(m,c)$ as follows:
\begin{align}
f(m,c) & = \sum_{k=0}^{\infty} \frac{(km+1)^{k-1} c^k}{m^k k!} e^{-(kc+c/m)} \\
           & = e^{-c/m}\sum_{k=0}^{\infty} \frac{(km+1)^{k-1} c^k}{m^kk!} e^{-kc} \\
           & = e^{-c/m}\sum_{k=0}^{\infty}\frac{(km+1)^{k-1}}{m^k k!}(ce^{-c})^k.
\end{align}
Let $g(m,x)$ be the following power series:
\begin{align}
g(m,x)&=\sum_{k=0}^{\infty} (km+1)^{k-1}/(m^k k!)x^k.
\end{align}
One can easily show that the radius of convergence of $g(m,x)$ is
$\frac{1}{e}$.
Also using the Sterling formula $k!=(1+o_k(1))\sqrt{2\pi k}k^k/e^k$,
one can show that $g(m,x)$ also absolutely converges at $x=\pm
\frac{1}{e}$. Thus $g(m,x)$ converges for all $x\in [-1/e, 1/e]$.

Note that for all real $c>0$, $ce^{-c}\leq 1/e$. This implies the
convergence of the infinite series in $f(m,c)$. The continuity of
$f(m,c)$ is deduced from the the continuity of $g(m,x)$ for all $m>0$
and all $x\in [-1/e,1/e]$.
\qed

\subsection{Labeled trees and rooted $m$-forest}
Cayley's formula states that for any positive integer $n$, the number of trees on $n$ labeled vertices is $n^{n-2}$.

A rooted forest on $[n]$ is a graph on the vertex set $\{1,2,\ldots,n\}$ for which every connected component is a rooted tree. It is well-known that there are $(n+1)^{n-1}$ rooted forests on $[n]$.

A rooted $m$-forest is a rooted forest on vertices $\{1,2,\ldots,n\}$ with edges colored with the colors $0,1,\ldots,m-1$. There is no additional restriction on the possible colors of the edges. This definition is due to Stanley, see \cite{yan1997, stanley1998}. It is easy to see by standard enumerative arguments that there are $(mn+1)^{n-1}$ rooted $m$-forests on $[n]$.

\begin{lemm} \label{key-lemma} For any two positive integers $m$ and $n$, we have
\begin{align}
\sum_{k_1+k_2+\cdots+k_m=n}\binom{n}{k_1,k_2,\ldots,k_m}\prod_{i=1}^m(k_im+1)^{k_i-1}=(n+1)^{n-1}m^n.
\end{align}
\end{lemm}

\proof Consider the labeled trees on $n+1$ vertices $\{0,1,2,\ldots,n\}$ with edges colored with the colors
$0,1,\ldots,m-1$. It is clear that there are $(n+1)^{n-1}m^n$ such trees. Each labeled tree can be considered as
a rooted tree with the root $0$. By deleting all the edges  away from the root, taking the children of $0$ as roots, and grouping the subtrees by the colors of the edges linked with $0$, we get a collection of rooted $m$-forests.
This completes the proof.
\qed

Following Yan \cite{yan2001}, let $B_n$ be the set of all sequences $(S_1, S_2, \ldots, S_q)$ of length $q$ such that
\begin{itemize}
\item[(1)] each $S_i$ is a rooted $p$-forest,

\item[(2)] $S_i$ and $S_j$ are disjoint if $i\neq j$, and

\item[(3)] the union of the vertex sets of $S_1, S_2, \ldots, S_a$ is $[n]$.
\end{itemize}

Alternatively, $B_n$ can be considered as the set of rooted forests on $[n]$ with root vertices colored with $0,1,\ldots q-1$
and non-root vertices colored with $0,1,\ldots,p-1$. The elements of
$B_n$ are called {\em rooted $(p,q)$-forests}.

As remarked by Yan, the cardinality of the set $B_n$ is $q(q+np)^{n-1}$. This result can be obtained by using a simple generalization of the Pr\"ufer code on rooted forests \cite[Chap. 5.3]{stanley1999}.

\begin{lemm} \label{key-lemma-e}  For any positive integers $n$, $p$,
  and $q$, we have
$$\sum_{j_1+j_2+\cdots+j_p=n} \!\!\!\!
\binom{n}{j_1,j_2,\ldots,j_p} \prod_{t=1}^p q(j_tp+q)^{j_t-1}= \!\!\!\!
\sum_{i_1+i_2+\cdots+i_q=n}
\!\!\!\!
 \binom{n}{i_1,i_2,\ldots,i_q}p^n\prod_{t=1}^q (i_t+1)^{i_t-1}.$$
\end{lemm}

\proof Let $A_n$ denote the set of rooted forests on $[n]$ with vertices colored with $0,1,\ldots,p-1$ and
root vertices also colored with $0,1,\ldots q-1$. (Namely, each root
will have a pair of two colors $(i,j)$, and each non-root will have
only one color $i$, for some $i\in \{0,1,\ldots, p-1\}$ and $j\in
\{0,1,\ldots, q-1\}$). There are two ways to count $A_n$:

(1) For each colored forest in $A_n$, we group trees according to the
first color of their roots. There are $\binom{n}{j_1,j_2,\ldots,j_p}$ ways to divide the vertex set $[n]$ into $p$ blocks $A_{n1},A_{n2},\ldots,A_{np}$ with respective size $j_1,j_2,\ldots,j_p$. Each
$A_{nt}$ is a rooted $(p,q)$-forest. By Yan's result, the number of
choices of $A_{nt}$ with size $j_t$ is given by
$q(j_tp+q)^{j_t-1}$. Thus the total number of objects in $A_n$ is
$$\sum_{j_1+j_2+\cdots+j_p=n}\binom{n}{j_1,j_2,\ldots,j_p} \prod_{t=1}^p q(j_tp+q)^{j_t-1}.$$

(2) For each colored forest in $A_n$, we group trees according to the
second color of their roots. There are $\binom{n}{j_1,j_2,\ldots,j_q}$ ways to divide the vertex set $[n]$ into $q$ blocks $A'_{n1},A'_{n2},\ldots,A'_{nq}$ with respective size $i_1,i_2,\ldots,i_q$. Each
$A'_{nt}$ is a rooted $(p,p)$-forest, which has
$p(i_tp+p)^{i_t-1}=p^{i_t}(i_t+1)^{i_t-1}$ choices. Thus the total number of objects in $A_n$ is
$$\sum_{i_1+i_2+\cdots+i_q=n} \binom{n}{i_1,i_2,\ldots,i_q}\prod_{t=1}^q p^i_t(i_t+1)^{i_t-1}.$$
The identity follows since both sides count the same set $A_n$.
\qed

\section{Two elementary proofs of Theorem \ref{main}}

\subsection{A Combinatorial Proof}



\proof 
Note that $f(m,c)=e^{-c/m}g(m, ce^{-c})$. To show $f(m,c)^m=f(1,c)$,
it is sufficient to show $g(m,x)^m=g(1,x)$ for all $x\in [0,1/e]$, 
namely
\begin{align}
 \left(\sum_{k=0}^{\infty} (km+1)^{k-1}/(m^k k!)(x)^k\right)^m=\sum_{k=0}^{\infty} (k+1)^{k-1}/k!x^k.
 \end{align}

Taking the coefficients of $x^n$ on both sides, we obtain that
\begin{align}
\sum_{k_1+k_2+\cdots+k_m=n}\prod_{i=1}^m(k_im+1)^{k_i-1}/(m^{k_i}k_i!)=(n+1)^{n-1}/n!,
\end{align}
which can be written as
\begin{align}
\sum_{k_1+k_2+\cdots+k_m=n}
\binom{n}{k_1,k_2,\ldots,k_m}\prod_{i=1}^m(k_im+1)^{k_i-1}=(n+1)^{n-1}m^n.
\end{align}
By Lemma \ref{key-lemma}, the proof is complete.
\qed

\subsection{A Power-Series Proof}
Here we use the following  version of the
well-known Lagrange inversion formula \cite{Lagrange}:

\vspace{.1in}

\noindent
{\bf Lagrange inversion  formula}
\\
{\it Suppose that $z$ is a function of $x$ and $y$ in terms of another
analytic function $\phi$ as follows:
\[ z = x + y \phi(z). \]
Then $z$  can be written as a power series in $y$   as follows:
\[ z = x + \sum_{k=1}^{\infty} \frac{y^k}{k!} D^{(k-1)}
 \phi^k(x) \]
where $D^{(t)}$ denotes the $t$-th derivative.
}

\begin{lemm}\label{l2}
 The function $f(m,c)$ satisfies Equation \eqref{eq:3}.
\end{lemm}
\begin{proof}
 Write Equation \eqref{eq:3} as
$$z=e^{-c/m}e^{cz^m/m}.$$
Applying the Lagrange inversion  formula with
$x=0$, $y=e^{-c/m}$, and $\phi(z)=e^{cz^m/m}$.
Note that $\phi^k(z)=e^{kcz^m/m}=\sum_{j=0}^\infty(kc/m)^j z^{mj}/j!$.
Thus,
$$
D^{(k-1)}
 \phi^k(0)=
 \begin{cases}
 (k-1)!  (kc/m)^j/j! & \mbox{ if }k=mj+1\\
  0      &\mbox{ otherwise.}
 \end{cases}
$$
We have
\begin{align*}
z&=x + \sum_{k=1}^{\infty} \frac{y^k}{k!} D^{(k-1)}
 \phi^k(x)\\
&=\sum_{j=0}^\infty \frac{e^{-c(mj+1)/m}}{(mj+1)!} \frac{(mj)!
  ((mj+1)c/m)^j}{j!}\\
&=  \sum_{j=0}^\infty \frac{(mj+1)^{j-1}}{m^j j!}c^je^{-c(mj+1)/m}\\
&=f(m,c).
\end{align*}
\end{proof}

\begin{proof}{Proof of Theorem \ref{main}:}
Let $z=f(m,c)$.
By Lemma \ref{l2}, we have
$$z=e^{-(1-z^m)c/m}.$$
Thus
$$z^m=e^{-(1-z^m)c}.$$
Applying Lemma \ref{l2} again, we get
$$z^m=f(1,c).$$
\end{proof}

\section{Extending Theorem \ref{main} to Real $m$}
We will extend Theorem \ref{main} to real $m$ as follows.
\begin{theo}\label{real}
For any two positive reals $m$ and $c$, we have
$$f(m,c)^m=f(1,c).$$
\end{theo}
\begin{proof}
Observe that $f(m,c)$ is continuous on $m$. It is sufficient to 
show that $f(m,c)^m=f(1,c)$ holds for rational $m>0$. Equivalently, it suffices to show that $g(m,x)^m=g(1,x)$, namely
\begin{align}
\left(\sum_{k=0}^{\infty} (km+1)^{k-1}/(m^k k!)(x)^k\right)^m=\sum_{k=0}^{\infty} (k+1)^{k-1}/k!x^k.
\end{align}

Suppose that $m=p/q$. Then
\begin{align}
\left(\sum_{k=0}^{\infty} (kp+q)^{k-1}q/(p^k k!)(x)^k\right)^p=\left(\sum_{k=0}^{\infty} (k+1)^{k-1}/k!x^k\right)^q.
\end{align}
Comparing the coefficients of $x^n$ on both sides, we obtain
$$\sum_{j_1+j_2+\cdots+j_p=n}\prod_{t=1}^p (j_tp+q)^{j_t-1}q/(p^{j_t} j_t!)=
\sum_{i_1+i_2+\cdots+i_q=n}\prod_{t=1}^q (i_t+1)^{i_t-1}/(i_t!),$$
or equivalently
$$\sum_{j_1+j_2+\cdots+j_p=n}\!\!\!\!
\binom{n}{j_1,j_2,\ldots,j_p} \prod_{t=1}^p q(j_tp+q)^{j_t-1}=\!\!\!\!
\sum_{i_1+i_2+\cdots+i_q=n} \!\!\!\!
 \binom{n}{i_1,i_2,\ldots,i_q}p^n\prod_{t=1}^q (i_t+1)^{i_t-1}.$$
By Lemma \ref{key-lemma-e}, this completes the proof.  
\end{proof}


\end{document}